\documentclass[11pt, oneside]{amsart}   	
\usepackage[margin=1in]{geometry}                		
\usepackage{graphicx}					
\usepackage{amssymb}
\usepackage{amsmath}
\usepackage{amsthm}
\usepackage{enumerate}
\usepackage{color}
\usepackage{wrapfig}
\usepackage{multicol}
\usepackage{caption}
\usepackage{subcaption}

\newtheorem{theorem}{Theorem}[section]

\newtheorem{lemma}[theorem]{Lemma}
\newtheorem{proposition}[theorem]{Proposition}
\newtheorem{question}[theorem]{Question}

\newenvironment{definition}[1][Definition:]{\begin{trivlist}
\item[\hskip \labelsep {\bfseries #1}]}{\end{trivlist}}

\newenvironment{example}[1][Example:]{\begin{trivlist}
\item[\hskip \labelsep {\bfseries #1}]}{\end{trivlist}}

\newcommand{\R}{\mathbb R}
\newcommand{\C}{\mathcal C}
\newcommand{\V}{\mathcal V}
\newcommand{\U}{\mathcal U}

\newcommand{\od}{\stackrel{\text{def}}{=}}
\newcommand{\closure}{\overline}
\DeclareMathOperator{\supp}{supp}
\DeclareMathOperator{\interior}{int}

\DeclareMathOperator{\trim}{trim}
\DeclareMathOperator{\inflate}{inflate}

\title{Sparse neural codes and convexity}
\begin{document}

\author{R. Amzi Jeffs, Mohamed Omar, Natchanon Suaysom, Aleina Wachtel, Nora Youngs}
\address{Department of Mathematics.  Harvey Mudd College, Claremont, CA 91711}
\email{rjeffs@g.hmc.edu, omar@g.hmc.edu, nsuaysom@g.hmc.edu, awachtel@g.hmc.edu, nyoungs@g.hmc.edu}
%
%\author{Mohamed Omar}
%\address{Department of Mathematics.  Harvey Mudd College, Claremont, CA 91711}
%\email{omar@g.hmc.edu}
%
%\author{Natchanon Suaysom}
%\address{Department of Mathematics.  Harvey Mudd College, Claremont, CA 91711}
%\email{nsuaysom@g.hmc.edu}
%
%\author{Aleina Wachtel}
%\address{Department of Mathematics.  Harvey Mudd College, Claremont, CA 91711}
%\email{awachtel@g.hmc.edu}
%
%\author{Nora Youngs}
%\address{Department of Mathematics.  Harvey Mudd College, Claremont, CA 91711}
%\email{nyoungs@g.hmc.edu}


\begin{abstract} 
Determining how the brain stores information is one of the most pressing problems in neuroscience.  In many instances, the collection of stimuli for a given neuron can be modeled by a convex set in $\mathbb{R}^d$.  Combinatorial objects known as \emph{neural codes} can then be used to extract features of the space covered by these convex regions.  We apply results from convex geometry to determine which neural codes can be realized by arrangements of open convex sets. We restrict our attention primarily to sparse codes in low dimensions.  We find that intersection-completeness characterizes realizable 2-sparse codes, and show that any realizable 2-sparse code has embedding dimension at most $3$. Furthermore, we prove that in $\R^2$ and $\R^3$, realizations of 2-sparse codes using closed sets are equivalent to those with open sets, and this allows us to provide some preliminary results on distinguishing which 2-sparse codes have embedding dimension at most $2$.
\end{abstract}

\maketitle

\section{Introduction}

A fundamental problem in convex geometry is to understand the intersection behavior of convex sets.  Classical theorems in this area include Helly's theorem and its many variations, which show that the presence of lower order intersections of convex sets in $\R^d$ can enforce intersections of higher order \cite{Matousek}.  Recent work of Tancer \cite{tancer} on the representability of simplicial complexes provides a sharp bound on the dimension in which intersection patterns of convex sets can be realized. We consider the problem of simultaneously realizing intersection patterns along with other relationships between convex sets, such as containment.  This problem is motivated by one of the challenges of mathematical neuroscience: determining how the structure of a stimulus space is represented in the brain.

Neurons respond to stimuli in an environment; the set of all such stimuli is called the {\bf stimulus space} $X$.  Usually, we consider $X\subset \R^d$.  If we are considering data from $n$ neurons $\{1,...,n\}$ which respond to stimuli in $X$, the {\bf receptive field} for neuron $i$ is the subset $U_i$ of the stimulus space $X$ for which neuron $i$ is highly responsive. Throughout this article, we assume the sets $U_i$ are convex.  Indeed, experimental data on many types of neurons, such as place cells \cite{OKeefeDostrovsky} or orientation-tuned neurons \cite{hubelwiesel}, make it evident that receptive fields are often very well approximated by convex sets. Hence, for such neurons, the regions of stimulus space in which multiple neurons fire can be modeled by intersections of convex sets.  Because of this, the mathematical theory developed by Helly, Tancer, and others can inform us about the possible arrangements of receptive fields in a given dimension.

Helly's theorem, however, cannot inform us about all types of receptive field arrangements.  For example, if  $U_i$, $U_j$ are receptive fields which properly intersect, the neural data will differentiate between $U_i\subset U_j$ and  $U_i\not\subset U_j$, but Helly's theorem merely notes that $U_i$ and $U_j$ intersect. We thus go beyond the usual scope of convex geometry to consider the problem of finding arrangements of convex sets which fully realize the information present in the neural data, including containments. This problem was posed originally in \cite{Youngs2013}.  In order to address this issue, we must first describe how neural data is represented mathematically.

\begin{definition}\label{def:neuralcode}
A {\bf neural code} on $n$ neurons is a set of binary firing patterns $\C\subset\{0,1\}^n$, representing neural activity. Elements of $\C$ are referred to as {\bf codewords}.
\end{definition}

The firing of a neuron is an all-or-nothing event, and so a codeword $c \in \C$ represents a data point in which a specific set of neurons are simultaneously firing; with neuron $i$ active if $c_i=1$ and inactive if $c_i=0$.  For instance, the codeword $0011$ represents a data point at which neurons 3 and 4 were active, while neurons 1 and 2 were not.  In the receptive field context, the presence of this codeword in $\C$ indicates that $(U_3 \cap U_4 ) \backslash (U_1\cup U_2) \neq \emptyset$.

\begin{definition}\label{def:realization}

Let $\U = \{U_1,\ldots,U_n\}$ be a collection of sets in $\R^d$.
The {\bf associated neural code} $\C(\U) \subset \{0,1\}^n$ is the set of firing patterns representing the regions in the arrangement: 
$$\C(\U) \od \left \{ c \in \{0,1\}^n \ \bigg | \ \left(\bigcap_{c_i = 1} U_i \right) \setminus \left(\bigcup_{c_j = 0} U_j \right) \neq \emptyset \right \}.$$
\end{definition}

Any collection of sets $\U$ in $\R^d$ can give rise to an associated neural code.  However, as we have mentioned, the receptive fields $U_i$ are generally presumed to be convex. One of our main motivating example is that of place cells, whose receptive fields are generally seen to be convex as explained in \cite{Youngs2015}. 
% Furthermore, convex sets are necessarily connected, and it has been shown that any code may be realized in $\R^1$ if the sets are not required to be connected \cite{Youngs2013}. 
  We additionally assume the receptive fields $U_i$ are open, since by restricting to open sets, we force all sets in our realization to be full-dimensional; furthermore, their intersections, if nonempty, must also be full-dimensional, which corresponds to non-degenerate cases in reality. These assumptions are consistent with the literature \cite{Youngs2015,  Youngs2013, Shiu2015}. However, many of our proofs will require that we shift between closed and open convex sets which are associated to the same code. We therefore make the following definition: \\

\begin{definition}
If  $\U=\{U_1,...,U_n\}$ is a collection of open (respectively, closed) convex sets in $\R^d$ for which $\C = \C(\U)$, then we say that $\C$ is \textbf{open (closed) convex realizable in }$\R^d$, and that $\U$ is an \textbf{open (closed) convex realization} of $\C$.

Then, for any code $\C$, we define $d(\C)$ to be the minimum dimension $d$ such that $\C$ has an open convex realization in $\R^d$, if such a dimension $d$ exists.  If $\C$ is not realizable with open convex sets in {\it any} dimension, we say $d(\C) = \infty$. Such codes do exist; see Figure~\ref{fig:pathargument}.
\end{definition}
%
%One of our main results describes conditions under which we can move from a  will occasionally be necessary to reference a {\bf closed} convex realization of a code. This definition is equivalent to the above, with open sets replaced by closed.

\begin{figure}[h]
\includegraphics[scale=.25]{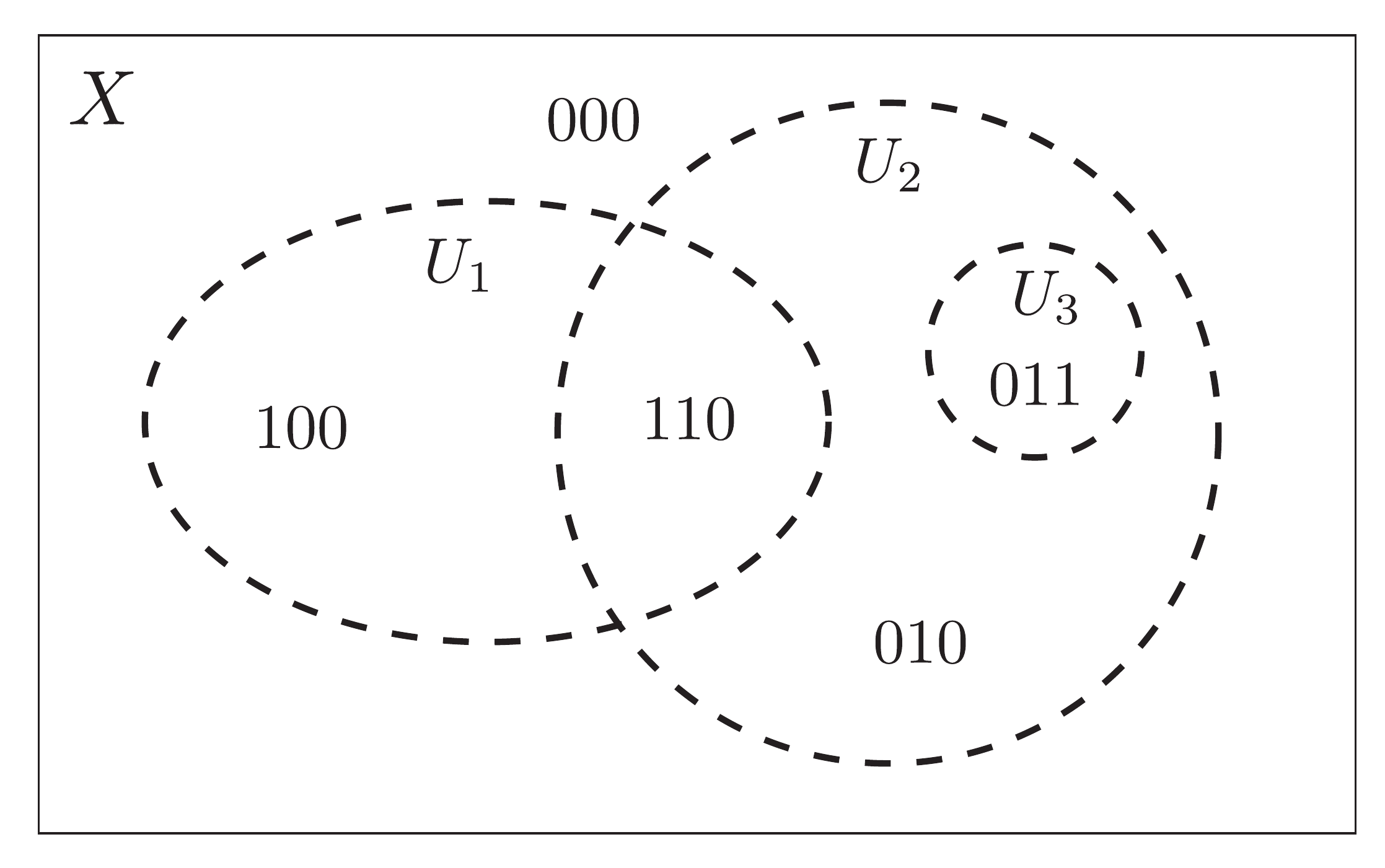}
\caption{An open convex realization of the code $\C = \{000, 100, 010, 110, 011\}$ in $\R^2$, with each region labelled with its corresponding codeword.  This shows $d(\C)\leq 2$.  It can be shown that, in fact, $d(\C)=1$. } \label{fig:realization}
\end{figure}

\begin{figure}[h]
\includegraphics[scale=.25]{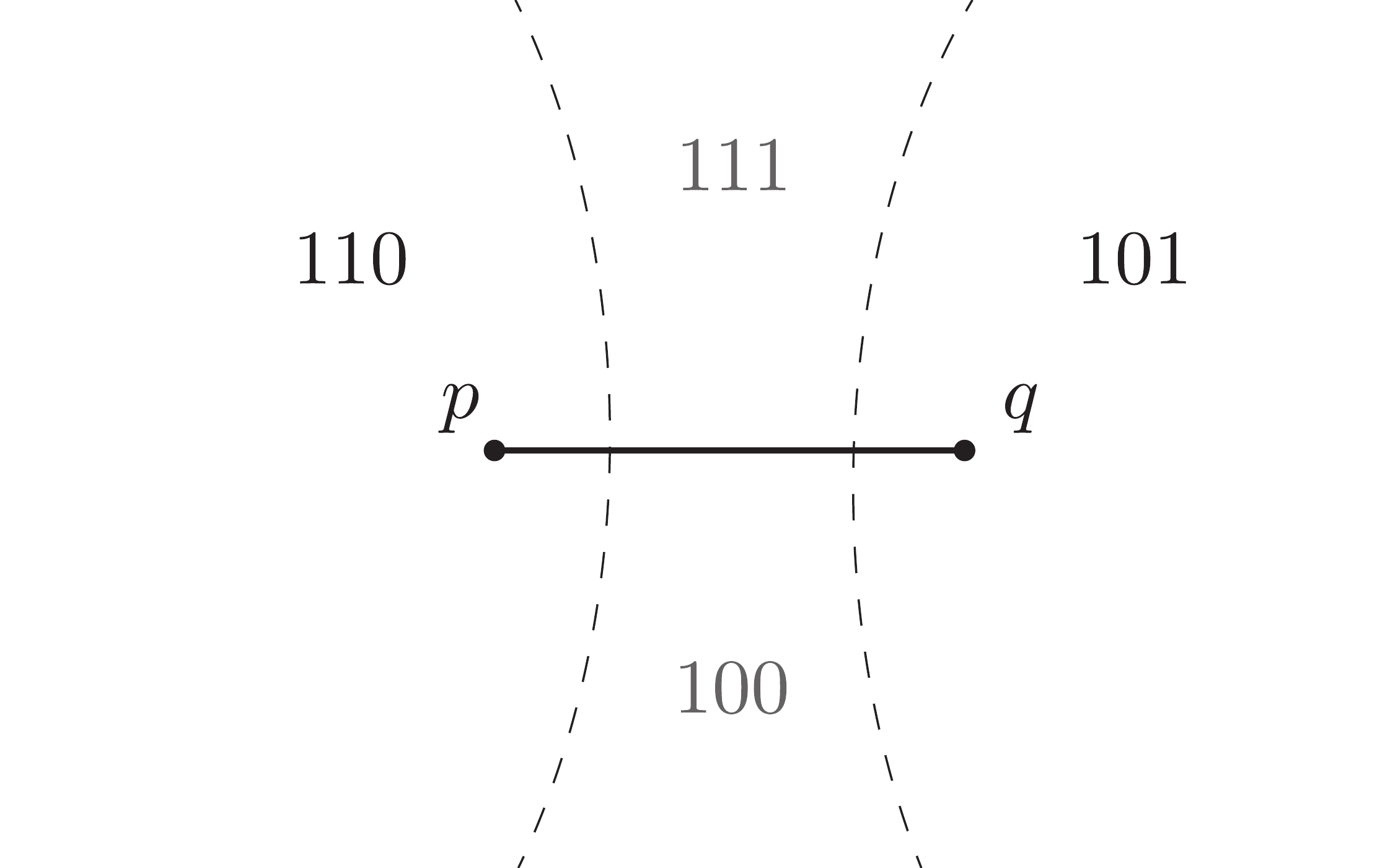}
\caption{ The code $\C = \{000, 010, 001, 110,  101\}$ is not open convex realizable in $\R^d$ for any $d<\infty$. If it were, we could pick points $p\in (U_1\cap U_2)\backslash U_3$ and $q\in (U_1\cap U_3)\backslash U_2$.  The line segment $\overline{pq}$ is contained in $U_1$ by convexity; to move from $p$ to $q$ along $\overline{pq}$, we must leave $U_2$ and enter $U_3$.  If we leave $U_2$ before entering $U_3$ that would indicate the presence of codeword 100, which is not in the code; if we enter $U_3$ before leaving $U_2$ that would indicate the codeword 111, which is not in the code. Since all sets are open, these are the only possibilities. } \label{fig:pathargument}
\end{figure}

% In examining codes and their realizations, we often need to isolate the sections of a realization which correspond to specific codewords. The following definition captures this notion precisely. \\

%\begin{definition}\label{def:codewordregion}
%Let $\C$ be a code with a realization $\U$. If $c$ is a codeword then the \textbf{region corresponding to $c$} or \textbf{codeword region for $c$}, written $V_c$, is the set \[
%V_c\od \left(\bigcap_{c_i = 1} U_i \right)\setminus \left(\bigcup_{c_i = 0} U_i \right).
%\]
%\end{definition}

%\begin{figure}[h]
%\includegraphics[scale=.25]{realization_example_with_codewords_colored}
%\caption{Realization of $\C$ with labeled codeword regions}\label{fig:codewordregion}
%\end{figure}
%

\begin{definition}  The {\bf support} of a vector $c\in\{0,1\}^n$, denoted $\supp(c)$, is the set of indices of value 1, or the set of all firing neurons:
$$\supp (c) \od \{i \mid c_i =1 \}$$ 
The support of a code $\C\subset \{0,1\}^n$ is the set of the supports of its codewords: $$\supp (\C) \od \{\supp (c) \mid c \in \C \}.$$ 
\end{definition}

We assume that there are instances when none of the neurons of interest are firing; hence, we will always assume that the codeword $00\cdots 0$ is present in any code.\\

\begin{example}
Let $\C = \{000, 101, 110, 111\}$. Then $\supp(101)=\{1,3\}$, $\supp(111)=\{1,2,3\}$, and $\supp(\C)=\{\emptyset,\{1,3\},\{1,2\},\{1,2,3\}\}$.\\
\end{example}

%Experimental data shows that receptive field codes are fairly sparse.  

Recent work, for example that of Miesenb\"ock et al. \cite{SparseCoding} shows the utility and importance of sparsity in neural codes.  For practical reasons, our definition of `sparse' differs slightly from the usual low average weight definition often used in coding literature (see for example \cite{Walker2013}) and we use instead a low maximum weight definition, as defined here.

\begin{definition}\label{def:sparse}
A code $\C$ is \textbf{$k$-sparse} if $|\supp(c)| \le k$ for all $c\in \C$.
\end{definition}

We begin the program of studying $k$-sparse code by focusing on $2$-sparse codes, where there is already rich mathematics to be found.  Our fundamental motivating questions are the following:

\begin{question}
Which 2-sparse codes are open convex realizable? 
\end{question}

\begin{question}
If $\C$ is an open convex realizable 2-sparse code, what is its minimum embedding dimension $d(\C)$?
\end{question}

In Section 2, we will answer the first question entirely. In Theorem~\ref{thm:R3} we show that 2-sparse codes are open convex realizable exactly when they are closed under intersection.  In the process, we show in Lemma~\ref{lem:openvsclosed} that for such codes it is equivalent to find a closed convex realization, as it may be transformed to an open convex realization in $\R^2$ or $\R^3$.   It immediately follows from this and work of Tancer \cite{tancer} that any 2-sparse code has a convex open realization in $\R^3$.  In Section 3, we exhibit a class of 2-sparse codes with $d\leq 2$, and  as well as a class with $d=3$.

\section{Equivalence with graphs}~\label{sec:topological}
This section is dedicated to proving Theorem \ref{thm:R3}, which establishes that a $2$-sparse code is realizable exactly when its support is closed under intersection, and for such codes $\C$, $d(\C) \leq 3$.  An important construct in proving this theorem is that of simplicial complexes.

In order to prove this theorem, we make use of the simplicial complex of a code, which is introduced below.

\begin{definition}\label{def:simplicialcomplex}
A {\bf simplicial complex} on a finite set $S$ is a family $\Delta$ of subsets of $S$ such that if $X\in \Delta$ and $Y \subseteq X$, then $Y \in \Delta$.
\end{definition}

Often, the set $S$ under consideration will be $[n] = \{1,...,n\}$.  In a situation where $S=\{v_1,...,v_n\}$, we will typically refer to any set in a simplicial complex on $S$ by its set of indices.

\begin{definition} The {\bf simplicial complex of a code} $\C$ is the smallest simplicial complex containing $\supp(\C)$; this is denoted $\Delta(\C)$.
%\end{definition}
%
%We will often want to refer to the sets in a simplicial complex below a fixed size. This is captured by the following definition.
%
%
%\begin{definition}\label{def:kskeleton}
The \textbf{$k$-skeleton} of a simplicial complex $\Delta$ is the simplicial complex $\Delta_k$ given by the collection of sets in $\Delta$ of size at most $k + 1$.
\end{definition}

\begin{figure}[h]
\includegraphics[scale=.45]{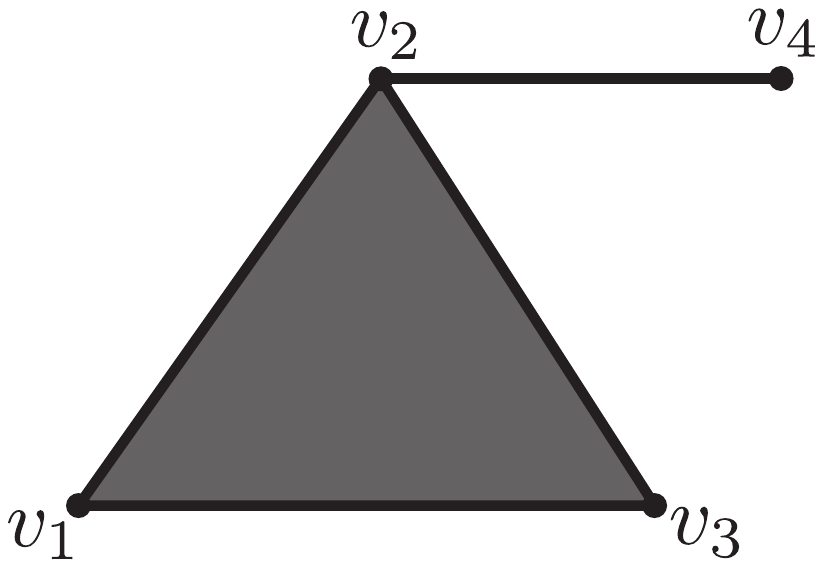}\hspace{1in} \includegraphics[scale=.45]{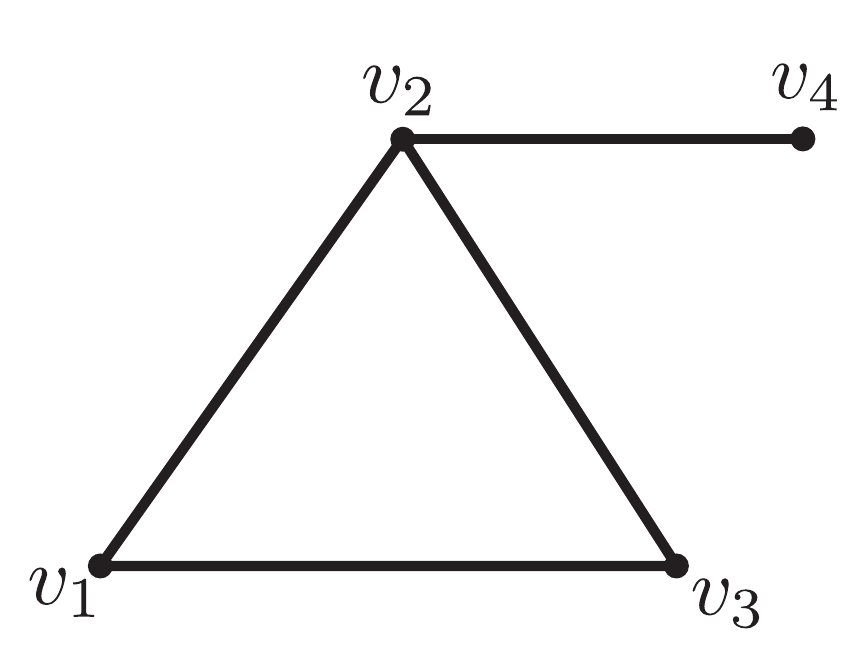}
\caption{At left, a geometric representation of  simplicial complex on $S=\{v_1,v_2,v_3,v_4\}$ with $\Delta = \{ \emptyset, \{1\}, \{2\}, \{3\}, \{4\}, \{1,2\}, \{1,3\}, \{2,3\}, \{2,4\}, \{1,2,3\} \}$. At right, a geometric representation of the 1-skeleton of $\Delta$} 
\end{figure}

%\begin{figure}[h]
%\includegraphics[scale=.45]{simplical_complex_1skeleton}
%\caption{The 1-skeleton of the simplicial complex $\Delta$ from Figure %PUT IN FIGURE NUMBER
%} 
%\end{figure}

If $\C$ is 2-sparse, then $\Delta(\C)$ consists only of $0, 1$, and $2$ element sets.  We can therefore think of $\Delta(\C)$ as a graph, with 1-element sets corresponding to vertices and $2$-element sets as edges between them.  Note that since $\Delta(\C)$ is a simplicial complex, if $\{i,j\}\in \Delta(\C)$, then both $\{i\}$ and $\{j\}$ must be in $\Delta(\C)$ as well; hence this association is well-defined.  The formal relationship between $2$-sparse codes and graphs is captured by the following definition.

 \begin{definition}\label{def:codegraph}
 % Do we want to call our codes neural codes always? Or can we just say codes?
 Let $\C\subset\{0,1\}^n$ be a neural code. The \textbf{graph of $\C$}, denoted $G_\C$, is the graph whose vertex set is $[n]$, with $i$ adjacent to $j$ if $\{i,j\} \in \supp(\C)$.
  \end{definition} 
  
Note that $G_\C$ is essentially the 1-skeleton of $\Delta(\C)$.  In particular, for a 2-sparse code, $\Delta(\C)$ and $G_\C$ contain exactly the same information because $\Delta(\C)$ is equal to its 1-skeleton.

\begin{figure}[h]
\includegraphics[scale=1]{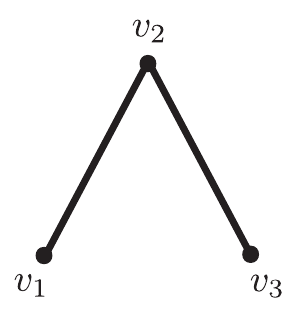}
\caption{The graph $G_{\C}$ for $\C = \{000, 100, 010, 110, 011\}$; see Figure \ref{fig:realization} for a realization of $\C$.}\label{fig:graph}
\end{figure}

  As we saw in Figure \ref{fig:pathargument}, there exist 2-sparse codes which are not convex in any dimension. The following lemma generalizes the obstruction presented in that figure.
  \begin{lemma}\label{lem:intersectioncomplete}
  Let $\C$ be a 2-sparse code. If $\C$ has a convex open realization in any dimension, then $\supp(\C)$ is closed under intersection.
  \end{lemma}
  \begin{proof}
Suppose $\C$ is a 2-sparse code with open convex realization $\U = \{U_1,...,U_n\}$.  Since $\C$ is 2-sparse, $|\supp(c)|\in \{ 0,1,2\}$ for every $c\in \C$.  
If $|\supp(c)|=1$ is at most $1$, then $\supp(c)\cap \supp(c') \in \supp(\C)$ because it is either $\emptyset$ or $\supp(c)$.   It then remains to show that $\supp(c) \cap \supp(c') = \{i\} \in \supp(\C)$ when $\supp(c) = \{i,j\}$ and $\supp(c')=\{i,k\}$ with $j \neq k$.  In this case $U_i \cap U_j$ and $U_i \cap U_k$ are nonempty so there exists points $p\in U_i\cap U_j$, $q\in U_i \cap U_k$.  Consider the line segment $\overline{pq}$ connecting $p$ and $q$. Since $U_i$ is convex, $\overline{pq}$ is contained in $U_i$. For each $m\in [n]\backslash\{i\}$, consider the set $L_m = \overline{pq}\cap U_i\cap U_m$; note that any two such sets are disjoint, and that $L_j$ and $L_k$ are nonempty. If the sets $\{L_m\}$ partition the line $\overline{pq}$, then this would disconnect $\overline{pq}$ in the subspace topology, but as $\overline{pq}$ is connected, this is impossible.  Thus, there must be some point on $\overline{pq}$ which is contained in $U_i$ only. The existence of this point implies $\{i\} \in \supp(\C)$ as desired.
\end{proof}

\noindent The conclusion of the previous lemma is in fact a characterization of 2-sparse codes with an open convex realization.  This is the content of the next theorem, one of our main results:

\begin{theorem}\label{thm:R3}
A 2-sparse code $\C$ has an open convex realization if and only if $\supp(\C)$ is closed under intersection.  Furthermore, if $\C$ is realizable then $d(\C) \leq 3$.
\end{theorem}

In order to prove Theorem~\ref{thm:R3}, we will repeatedly make geometric augmentations to existing realizations.  In order to make such augmentations without changing the underlying code, we must ensure that subset containment relations between sets are maintained.  In the 2-sparse case, the following definition encapsulates the key relationships that must be maintained:

\begin{definition}\label{def:2sparserelations}
Let $\U = \{U_1,\ldots, U_n\}$ be a collection of sets in $\R^d$. For any ordered pair $(U_i,U_j)$ we distinguish three possible relations between $U_i$ and $U_j$:
\begin{itemize}
\item\textbf{(Type A} \underline{Disjointness:} $U_i\cap U_j = \emptyset$; i.e., $\{i,j\} \not\in \supp(\C)$.
\item\textbf{(Type B)} \underline{Containment:} $U_j\subset U_i$; i.e., there exists a codeword $c\in \C(\U)$ so that $\{i,j\}\subseteq \supp(c)$ and any codeword whose support contains $i$ must also have $j$ in its support.
\item\textbf{(Type C)} \underline{Proper Intersection:} $U_i\cap U_j$ is nonempty and $U_j\setminus U_i$ is nonempty; i.e., there exist codewords $c_1,c_2 \in \C(\U)$ so that $\{i,j\}\subset \supp(c_1)$, $j\in \supp(c_2)$ and $i\notin \supp(c_2)$. 
\end{itemize}
\end{definition}

The Type A, Type B and Type C set relationships effectively characterize the structure of a 2-sparse code; indeed 2-sparse codes are completely determined by the pairwise relationships of the sets in a given realization of them. We explicitly state this in the following proposition:

\begin{proposition}\label{prop:possiblerelations}
Let $\U$ and $\U'$ be collections of sets in $\R^d$ so that $\C(\U)$ and $\C(\U')$ are both 2-sparse. Then $\C(\U) = \C(\U')$ if and only if for every ordered pair $(i,j)$ the relation between $U_i$ and $U_j$ is the same as the relation between $U_i'$ and $U_j'$. 
\end{proposition}

We now introduce the geometric underpinnings of the augmentations we will apply to realizations of codes.

\begin{definition}\label{def:shrink}
Given $\epsilon>0$ and $A \subset \R^d$, the \textbf{trim} of $A$ by $\epsilon$ is the set \[\trim(A,\epsilon) \od\interior \{ p\in \R^d \mid B_\epsilon(p)\subseteq A\}.\]

The \textbf{inflation} of $A$ by $\epsilon$ is the set \[\inflate(A,\epsilon) \od \{a + x \ | \ a\in A, \ x\in \R^d \text{ with } ||x||<\epsilon\}\]  If $\mathcal A = \{A_1,\ldots, A_n\}$ is a collection of sets, then $\trim(\mathcal A,\epsilon) = \{\trim(A_1,\epsilon),\ldots,\trim(A_n,\epsilon)\}$, and $\inflate(\mathcal A, \epsilon) = \{\inflate(A_1,\epsilon),\ldots,\inflate(A_n,\epsilon)\}$. 
\end{definition}

\begin{proposition}\label{prop:shrink}
For any convex set $A\subset \R^d$ and $\epsilon>0$,  the following statements hold:
\begin{enumerate}
\item $\trim(A,\epsilon)$ is an open convex set
\item  $\closure{\trim(A,\epsilon)}$ is contained in the interior of $A$
\item  $\inflate(A,\epsilon)$ is an open convex set
\end{enumerate}
\end{proposition}
\begin{proof}
For (1), we need only prove convexity, and we may assume $\trim(A,\epsilon)$ is nonempty.  Let $p$ and $q$ be points in $\trim(A,\epsilon)$; then $B_\epsilon(p)$ and $B_\epsilon(q)$ are contained in $A$, and hence so is the convex hull of their union. This convex hull contains the line segment $\overline{pq}$. For (2), note that $\closure{\trim(A,\epsilon)}\subseteq \trim(A, \epsilon/2) \subseteq \operatorname{int}(A)$.  Finally, (3) follows from the fact that $A$ is convex and $\{x \in \R^d \ |  \ ||x|| < \epsilon\}$ is open and convex.
\end{proof}

\begin{lemma}\label{lem:shrinking}
  Given a 2-sparse code $\C$ with an open convex realization $\U = \{U_1,\ldots,U_n\}$, there exists some $\epsilon > 0$ so that $\trim(\U,\epsilon)$ is also a realization of $\C$. 
\end{lemma}

\begin{proof}
Our plan is as follows: for each set $U_i$, we find an $\epsilon_i$ such that $\trim(U_i, \epsilon_i)\neq \emptyset$, and for each pair $\{i,j\}$ we find an $\epsilon_{ij}$ such that $\trim(\{U_i, U_j\}, \epsilon_{ij})$ preserves their relationship type (Type A, Type B or Type C).  We then let $\epsilon$ be the minimum of all $\epsilon_i$ and $\epsilon_{ij}$, and show that $\trim(\U,\epsilon)$ is a realization of the original code $\C$.

To start, for each $i$ with $U_i$ nonempty, there must be some point $p$ and $\delta_i>0$ with $B_{\delta_i}(p)\subset U_i$. Let $\epsilon_i = \delta_i/2$.  Let $\varepsilon_1 = \min_{i\in [n]} \epsilon_i$.  Now, for each pair $\{i,j\}$, we choose $\epsilon_{ij}$ depending on the relationship type between $U_i$ and $U_j$.
\begin{enumerate}
\item Type A: If $U_i\cap U_j = \emptyset$, set $\epsilon_{ij} = \min\{\epsilon_i, \epsilon_j\}$
\item Type B: If $U_i = U_j$, set $\epsilon_{ij} = \min\{\epsilon_i, \epsilon_j\}$. If $U_i\subsetneq U_j$, note that $U_j\backslash U_i$ has nonempty interior.  Thus there exists some point $p$ and some $\delta_{ij}>0$ with $B_{\delta_{ij}}(p)\subset U_j\backslash U_i$. Let $\epsilon_{ij} = \min\{\delta_{ij}/2, \epsilon_i\}$.
\item Type C: If $U_i\cap U_j\neq \emptyset$, but neither $U_i\subset U_j$ nor $U_j\subset U_i$ is true, note that $U_i\cap U_j$ is open and therefore there exists a point $p$ and $\epsilon'>0$ with $B_{\epsilon'}(p)\subset U_i\cap U_j$. There exists also points $p_i$, $p_j$ in $U_i\backslash U_j$, $U_j\backslash U_i$ respectively, with corresponding $\hat\epsilon$ and $\tilde \epsilon$ such that $B_{\hat\epsilon}(p_i) \subset U_i\backslash U_j$ and $B_{\tilde\epsilon}(p_j)\subset U_j\backslash U_i$.  Pick $\epsilon_{ij} = \min\{\epsilon_i, \epsilon_j, \hat\epsilon/2, \tilde\epsilon/2, \epsilon'/2\}$.
\end{enumerate}

Let $\varepsilon_2 = \min_{i,j} \epsilon_{ij}$, and finally, let $\epsilon = \min\{\varepsilon_1, \varepsilon_2\}$.  Since $\trim(\epsilon, U) \subset U$, and originally there were no triple intersections, by construction it is impossible for $\trim(\epsilon, \U)$ to have triple intersections.  Thus, $\C(\trim(\epsilon, \U))$ is still 2-sparse.  We now show that $\C(\trim(\epsilon, \U)) = \C$.

If the codeword with support $\{i,j\}$ is in $\C(\trim(\epsilon,\U))$, then $\trim(\epsilon,U_i)\cap \trim(\epsilon, U_j)\neq \emptyset$. As $\trim(\epsilon, U)\subset U$, this implies that $U_i\cap U_j\neq \emptyset$. Since $\C$ is 2-sparse, the codeword with support $\{i,j\}$ is in $\C$.  On the other hand, if the codeword with support $\{i,j\}$ is in $\C$, then $U_i\cap U_j\neq \emptyset$, and so we are in case $(1), (2)$, or $(3)$ above. By our choice of $\epsilon$, we ensure that in each case $\trim(U_i,\epsilon)\cap \trim(U_j,\epsilon) \neq \emptyset$, and hence (as the code is 2-sparse) the codeword with support $\{i,j\}$ is in $\C(\trim(\U,\epsilon))$.

If a codeword with support $\{i\}$ is in $\C(\trim(\U, \epsilon))$, then $\trim(U_i, \epsilon) \backslash \bigcup_{j\in [n], j\neq i} \trim(U_j,\epsilon) \neq \emptyset.$  We then know that $U_i\backslash \bigcup_{j\in [n], j\neq i}U_j \neq \emptyset$.  If it were not, then we would have $U_i\subset \bigcup_{j\in I}U_j$ for some index set $I$.  However,  this is impossible: if $|I| = 1$, then $U_i\subset U_j$, but then $\trim(U_i, \epsilon) \subset \trim(U_j, \epsilon)$.  If $|I|>1$, then  $U_i\subset \bigcup_{j\in I} U_j$.  But then the 2-sparsity of $\C$ means we would see the codewords $\{i,j\}$ and $\{i,k\}$ in $\C$ for $j,k\in I$ but not their intersection $\{i\}$, contradicting Lemma \ref{lem:intersectioncomplete}. Hence, the codeword with support $\{i\}$ is in $\C$.

Now, suppose a codeword with support $\{i\}$ is in $\C$, and let $J = \{j\ | \ U_i\cap U_j \neq \emptyset\}$.  If $|J|\leq 1$ then we are in case $(1), (2)$, or $(3)$ above, and by our choice of $\epsilon$ we know there is a codeword with support $\{i\}$ in $\C(\trim(\U, \epsilon))$.  If $|J|\geq 2$, let $j,k\in J$. Then by our choice of $\epsilon$, we know $\trim(U_i,\epsilon)\cap \trim(U_j, \epsilon)\neq \emptyset$ and $\trim(U_i,\epsilon)\cap \trim(U_k,\epsilon)\neq \emptyset$, and hence the codewords with supports $\{i,j\}$ and $\{i,k\}$ are in $\trim(\U, \epsilon)$. By Lemma \ref{lem:intersectioncomplete}, we know the codeword with support $\{i\}$ is also in $\C(\trim(\U, \epsilon))$.
\end{proof}

\begin{lemma}\label{lem:inflating}
Let $\C$ be a 2-sparse code with a closed convex realization $\V = \{V_1,\ldots, V_n\}$ in which every set is bounded. Then there exists some $\epsilon>0$ such that $\inflate(\V,\epsilon)$ is an open convex realization of $\C$. 
\end{lemma}

\begin{proof}
Consider the partial ordering on $\V$ given by set inclusion. We will use this ordering to inflate the sets in $\V$ iteratively (possibly by different $\epsilon$ factors) and then argue that we can obtain a uniform $\epsilon$ for which $\inflate(\V,\epsilon)$ is an open convex realization of $\C$.  In this iterative process, if $V_i=V_j$ for any $i \neq j$, we apply the process simultaneously to $V_i$ and $V_j$.  As such, it is sufficient for our proof to assume $V_i \neq V_j$ for any $i \neq j$.

To start, begin with a fixed index $i$ for which $V_i$ is maximal in $\V$ with respect to inclusion.  All sets in $\V$ are closed and bounded so for any $j$ with $V_j \cap V_j = \emptyset$, $V_i$ has positive distance $d_{i,j}$ to $V_j$.  Let $\displaystyle \delta_i=\min_{V_i \cap V_j = \emptyset} d_{i,j}$.  Now if there are $j,k\neq i$ with $V_j\cap V_k\neq\emptyset$, then $V_i$ has positive distance $d_{i,j,k}$ to $V_j\cap V_k$; take $\delta'_i$ to be the minimum of all such $d_{i,j,k}$. Furthermore, let $\delta''_i>0$ be such that for all $j$ with $V_j\not\subset V_i$, we have $V_j\not\subset\inflate(V_i,\delta''_i)$.  Finally, choose $\epsilon_i < \min\left\{\frac{\delta_i}{2}, \frac{\delta'_i}{2},\frac{\delta''_i}{2}\right\}.$

If we replace $V_i$ by $\closure{\inflate(V_i,\epsilon_i)}$, then the three subset relationship types for the ordered pairs $(V_i,V_j)$ where $j \neq i$ are maintained:

\begin{itemize}
\item Type 1: Disjointness is preserved since $\epsilon_i$ is at most half the distance from $V_i$ to any set disjoint from it.
\item Type 2: Containment is preserved since we are only making $V_i$ bigger.
\item Type 3: Proper intersection is preserved by our choice of $\epsilon_i$.
\end{itemize} 
By a similar argument, the subset relationship of the order pair $(V_j,V_i)$ for any $j \neq i$ is also preserved after replacing $V_i$ by $\closure{\inflate(V_i,\epsilon_i)}$.  Thus replacing $V_i$ by $\closure{\inflate(V_i,\epsilon_i)}$ yields a new realization of $\C$. 

For any subsequent step in our iterative process, choose a set $V_i \in \V$ for which every member of the set $\{V_j \in \V \ | \ V_j \supset V_i\}$ has already been inflated.  Choose $\epsilon_i$ in the same way as previously described with the additional caveat that if $V_i \subset V_j$ $\epsilon_i < \epsilon_j$.  A similar argument shows that replacing $V_i$ by $\closure{\inflate(V_i,\epsilon_i)}$ yields a new realization of $\C$. Once we have inflated every set in the realization we can let $\epsilon = \min_{i\in [n]} \epsilon_i$ and observe that  $\inflate(\U, \frac{\epsilon}{2})$ is an open convex realization of $\C$. 
\end{proof}

This  result allows us to prove the following useful fact. 

\begin{lemma}\label{lem:openvsclosed}
Let $\C$ be a 2-sparse code. Then $\C$ has an open convex realization in $\R^d$ if and only if $\C$ has a closed convex realization in $\R^d$.
\end{lemma}
\begin{proof}
$(\Rightarrow)$ Let  $\U$ be an open convex realization of $\C$.  Applying Lemma \ref{lem:shrinking}, there is an $\epsilon>0$ such that $\U' = \trim(\U, \epsilon)$ is an open realization of $\C$. Since the closure of each $U_i'$ is contained in $U_i$ (by Proposition \ref{prop:shrink}), $\U'$ is an open convex realization of $\C$ in which two sets intersect if and only if their closures do. Let $\V = \{\closure{U_1'},\ldots,\closure{U_n'}\}$. No triple intersections are present in $\V$ since these would correspond to triple intersections in $\U$.  Thus by Proposition \ref{prop:possiblerelations} it suffices to show that $\V$ preserves the relations between sets in $\U'$. Disjointness is preserved since sets in $\U$ intersect if and only if their closures do. Containment is preserved under taking closures. Lastly, proper intersection is preserved since if $U_i\setminus U_j$ is nonempty then there exist limit points of $U_i$ that are not limit points of $U_j$.  

$(\Leftarrow)$ Let $\V$ be a closed convex realization of $\C$. For every nonempty intersection $V_i\cap V_j$, let $p_{i,j}$ be a point in this intersection. Furthermore, if some set $V_i$ is not contained in any other $V_j$, let $p_i\in V_i\setminus \bigcup_{j\neq i}V_j$. Then set $V$ to be the convex hull of all these $p_i's$ and $p_{i,j}'s$. Replacing each $V_i$ by $V_i\cap V$ yields a realization of $\C$ in which every set is closed, convex, and bounded. Applying Lemma \ref{lem:inflating}, we obtain an open convex realization of $\C$ in $\R^d$.
\end{proof}
Although it may not be immediately clear from the proof, the condition that $\C$ is 2-sparse is necessary for Lemma \ref{lem:openvsclosed}. The 2-sparse condition is in fact best possible, since there exist 3-sparse codes which have closed convex realizations in $\R^2$, but for which open convex realizations exist only in $\R^3$ or higher. One such example is the code $$\C = \{0000,1000,0100,0010,0001, 1110, 1001, 0101, 0011\}$$ (see \cite{Shiu2015} for more details).  Even more strikingly, there exist codes with a closed convex realization in $\R^2$ that have no open convex realization in any dimension.  This emphasizes how special realizations of 2-sparse codes are.

We can now use the previous lemmas to relate the convexity of a 2-sparse code $\C$ to the convexity of its associated graph $G_\C$.  We first need a technical lemma.

\begin{lemma}\label{lem:boundarypt} Let $\U$ be an open convex realization of a 2-sparse code $\C$.  Then if $U_j\not\subset U_k$ for any $k\neq j$, there is a point $p\in \partial U_j \backslash \bigcup_{k\neq j} U_k$.
\end{lemma}

\begin{proof} Recall that for any set $U \subset \R^d$, $\partial U$ is the boundary of $U$.  Consider the sets $\{\partial U_j\cap U_k\}_{k\neq j}$.  These sets are disjoint: if not, then there exists $p \in (\partial U_j \cap U_k)\cap (\partial U_j \cap U_\ell)$. As $p\in U_k\cap U_\ell$, there exists $\epsilon>0$ with $B_\epsilon(p)\subseteq U_k\cap U_\ell$.  But then $B_\epsilon(p)\cap U_j\neq \emptyset$, as $p \in \partial U_j$, so $U_j\cap U_k\cap U_\ell\neq \emptyset$ contradicting that $\C$ is 2-sparse. \\
 
 Now, note that the disjoint sets $\{\partial U_j\cap U_k\}_{k\neq j}$ are open in the subspace topology with respect to $\partial U_j$, and hence they cannot partition $\partial U_j$ since $\partial U_j$ is connected. Thus, there exists $p\in \partial U_j\backslash \bigcup_{k\neq j} U_k$.
 \end{proof}

\begin{lemma}\label{lem:graphvscode}
Let $\C$ be a 2-sparse code and let $d\ge 2$. Then $\C$ has an open convex realization in $\R^d$ if and only if $\supp(\C)$ is closed under intersection and $G_\C$ has an open convex realization in $\R^d$.
\end{lemma}
\begin{proof}
$(\Rightarrow)$ We know from Lemma \ref{lem:intersectioncomplete} that if $\C$ has a realization then its support is closed under intersection. We will show that given a realization $\U$ of $\C$, we can construct a realization of $G_\C$. %Equivalently, we must construct an open convex realization of $\Delta(\C)$ from an open convex realization of $\C$. 
Since $\C$ is 2-sparse, we know $\C$ and $\Delta(\C)$ have the same $2$-element sets in them. We will show that we can adjust the realization of $\C$ to obtain any singletons $\{i\}$ which appear in $\Delta(\C)$ but not in $\C$. 

Let $\{i\}\in \Delta(\C)\backslash \supp(\C)$.  If there exist $j,k$ such that $\{i,j\}$ and $\{i,k\}$ are both in $\supp(\C)$, then as $\supp(\C)$ is closed under intersection, we know $\{i\}\in \supp(\C)$.  Thus, there is exactly one $j$ such that $\{i,j\}\in \supp(\C)$.

 Note right away that in the realization $\U$ we have $U_i \subset U_j$ since $\{i,j\}$ is the only set in the support where $i$ appears. It suffices to transform $\U$ so that $U_i$ and $U_j$ intersect, but $U_i$ also contains points not in any other set in the realization.
 
  If $U_j\subset U_i$, then $U_i=U_j$, and we can replace $U_j$ with an open ball properly contained in $U_i$ to obtain the desired result.  Otherwise, $U_j$ may intersect many other sets in the realization, but cannot be contained in them, since this would imply a triple intersection between the containing set, $U_j$, and $U_i$.  Apply Lemma \ref{lem:shrinking} to obtain $\epsilon>0$ for which $\U' = \trim(\U,\epsilon)$ is an open realization of $\C$. Define the sets $V_k = \partial U'_j \cap \overline{U'_k}$; note that each $V_k$ is closed. Furthermore, these sets are disjoint, since if $p\in V_k\cap V_\ell$, then $p\in U_j\cap U_k \cap U_\ell$ in the original realization which is impossible for a 2-sparse code. Since $\partial U_j'$ is connected and the $V_k$ are disjoint closed sets, $\bigcup_{k\neq j} V_k\subsetneq \partial U'_j$; let $p\in \partial U'_j\backslash \bigcup_{k\neq j} V_k$.  Then $p$ has positive distance to all sets $U'_k$ with $k\neq j$ so there is some $\epsilon'>0$ with $B_{\epsilon'}(p)\cap U'_k = \emptyset$ for all $k\neq j$.  Replacing $U_i'$ with $B_{\epsilon'}(p)$ will create a realization of a code $\C'$ with $\supp(\C') = \supp C \cup \{i\}$.  Repeating this step as many times as necessary, we obtain a realization of $\Delta(\C)$.
  
$(\Leftarrow)$ Suppose $\U$ is an open convex realization of $\Delta(\C)$.  Note that if $\{i,j\} \in \supp(\Delta(\C))$, it is also in $\supp(\C)$ since $\C$ is 2-sparse.  Now, suppose $\{i\} \in \supp(\Delta(\C))\backslash \supp(\C)$.  Then there is at most one $j\neq i$ such that $\{i,j\}\in \supp(\C)$ as $\C$ is closed under intersection. If there is such a $j$, replace $U_i$ with $U_i\cap U_j$ which is an open convex set; if there is no such $j$, then remove $U_j$ entirely. This gives a convex realization of $\Delta(\C)\backslash\{i\}$, and we can repeat this operation as many times as necessary to obtain a realization of $\C$.
\end{proof}

The above lemma can be summarized as follows: realizing a 2-sparse code and realizing its simplicial complex  are equivalent, as long as $\supp(\C)$ is closed under intersection. This equivalence is our main tool in proving Theorem \ref{thm:R3} and obtaining a complete classification of which 2-sparse codes are convex in $\R^3$. 

\begin{proof}[Proof of Theorem \ref{thm:R3}]
$(\Rightarrow)$ The fact that $\supp(\C)$ is closed under intersection follows directly from Lemma \ref{lem:intersectioncomplete}.

$(\Leftarrow)$ Since $\supp(\C)$ is closed under intersection, we know by Lemma \ref{lem:graphvscode} that it is equivalent to find a realization for $G_\C$. But $\C$ is 2-sparse, and so Lemma \ref{lem:openvsclosed} tells us that it suffices to find a closed convex realization for $G_\C$. The graph $G_\C$ is a 1-dimensional simplicial complex. An augmentation of construction of  Tancer \cite{tancer} (see the proof of Theorem 3.1 therein) leads to a realization of a $1$-dimensional simplicial complex in $\R^{3}$. 
%Nora: Thoughts on how to change the last sentence? 
This proves the desired result.
\end{proof}

Theorem~\ref{thm:R3} makes it very straightforward to check whether a 2-sparse code has an open convex realization in $\R^3$. The challenge that lies ahead is determining the minimal embedding dimension for a given 2-sparse code.  We begin investigating this problem in the next section.

\section{Realization results}

Throughout this discussion we will often refer to ``realizations'' of a graph. It is important to note that while a graph is the intersection graph of its realization, finding convex sets whose intersection graph is the graph of concern is not always sufficient. In particular, if a collection of convex sets has a triple with nonempty intersection then it is not a realization of a graph, since graphs only encode intersections of order two. 

In this section, we begin the program of classifying 2-sparse codes based on minimal embedding dimension.  Recall from Lemma~\ref{lem:graphvscode} that realizing a 2-sparse code $\C$ is equivalent to realizing its graph $G_\C$, so throughout this section we refer to realizing $G_\C$ rather than $\C$ itself.  Our main contribution is that it is difficult to determine graph theoretic criteria on $G_\C$ that predict $d(\C)$.  Topological properties inherent to $G_\C$ seem to potentially influence $d(\C)$: for instance in Proposition~\ref{prop:R2}, we observe $d(\C)=2$ if $G_\C$ is planar and in Proposition~\ref{prop:nonplanar} if $G_\C$ is not planar one can construct a related graph whose code has minimal embedding dimension $3$.  However, planarity does not strictly govern minimal embedding dimension, as all complete and complete bipartite graphs are realizable in $\R^2$. To start, our first proposition establishes that specific ubiquitous graphs have open convex realizations in $\R^2$.

\begin{proposition}\label{prop:R2}
The following graphs have an open convex realization in $\R^2$:
\begin{enumerate}
\item planar graphs,
\item the complete $k$-partite graph $K_{n_1,n_2,\ldots,n_k}$ with part sizes $n_1,n_2,\ldots,n_k$
\item any graph $G$ with vertex set $\{v_1,v_2,\ldots,v_n,u_1,\ldots,u_k\}$ where the induced subgraph on the vertices $v_1,v_2,\ldots,v_n$ is complete and $\{v_1,v_2,\ldots,v_n\} \supseteq N_G(u_k) \supseteq N_G(u_{k-1}) \supseteq \cdots \supseteq N_G(u_1)$.
\end{enumerate}
\end{proposition}

\begin{proof}
In all cases, we find a closed convex realization of the given graph $G$, which by Lemma~\ref{lem:openvsclosed} implies the existence of an open convex realization.  For (1), we first recall the Circle Packing Theorem 
%Nora: Do you know a good reference for the Circle Packing Theorem?
which says that for any planar graph $G$ with vertex set $\{v_1,\ldots,v_n\}$, there exist disjoint disks $C_1,C_2,\ldots,C_n$ in $\R^2$ such that $C_i$ is tangent to $C_j$ if and only if $v_i$ is adjacent to $v_j$, and $C_i \cap C_j = \emptyset$ otherwise.  See Figure \ref{fig:planargraphs} for an illustration of how these disks are constructed.  

\begin{figure}[h]\label{fig:circlepacking}
\centering
 \includegraphics[width=1.5in]{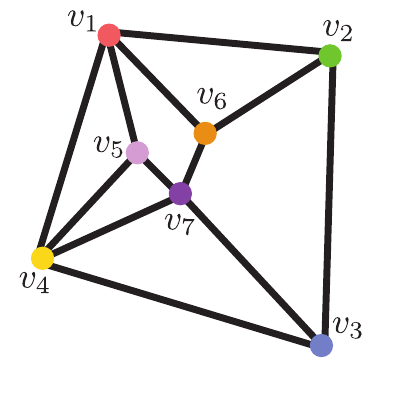} \hspace{1in}
  \includegraphics[width=1.8in]{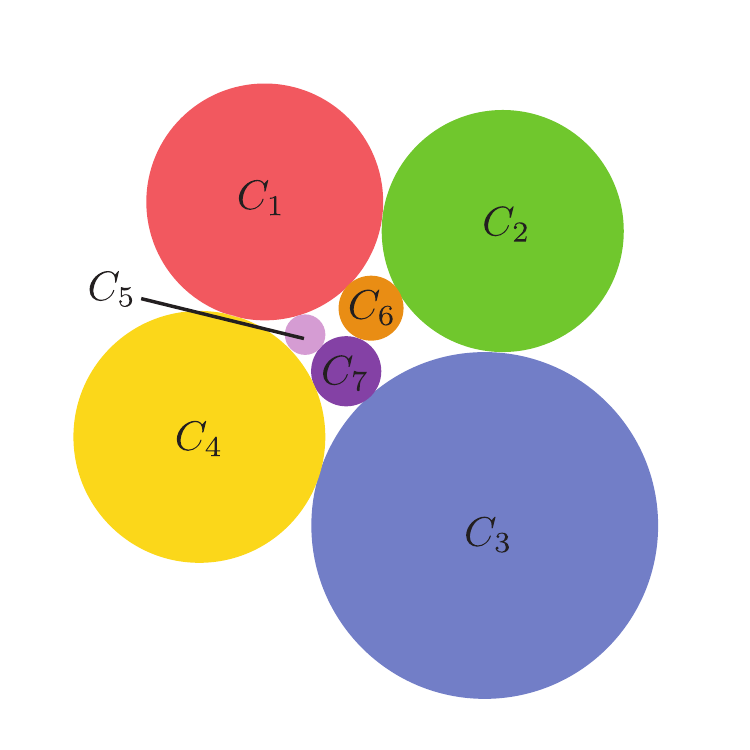}
\caption{A planar graph $G$ and the corresponding closed realization using the Circle Packing Theorem.}
\label{fig:planargraphs}
\end{figure}

For (2), we first find a realization for the complete graph $K_n=K_{1,1,\ldots,1}$ ($n$ copies of $1$ here).  Consider the line segments $\ell_1,\ell_2,\ldots,\ell_n$ where $\ell_i$ has endpoints $(i,0),(0,n+1-i)$, and observe that $\ell_i \cap \ell_j \neq \emptyset$ for any $i \neq j$.  Moreover, no three of these lines are concurrent.  This gives a closed convex realization of $K_n$.  Now to realize $K_{n_1,n_2,\ldots,n_k}$, start with a closed convex realization of $K_k$ as constructed in the realization of (2).  Replace each line segment $\ell_i$ with $n_i$ disjoint parallel translates of $\ell_i$ that are arbitrarily close in distance to $\ell_i$, and call these segments $s_{i1},s_{i2},\ldots,s_{in_i}$.  Observe that by construction, $s_{ij} \cap s_{ij'} = \emptyset$ for any $j \neq j'$.  Moreover, $s_{ij} \cap s_{i'j'} \neq \emptyset$ for $i \neq i'$ because $l_i \cap l_{i'} \neq \emptyset$ and $s_{ij}$ and $s_{i'j'}$ are arbitrarily close and parallel to $\ell_i$ and $\ell_{i'}$ respectively.  Moreover, if any three line segments $s_{ij},s_{i'j'},s_{i''j''}$ had a point in common, then $l_i,l_{i'},l_{i''}$ would, which they don't.  Hence the union of the sets $\{s_{i1},s_{i2},\ldots,s_{in_i}\}_{i=1}^k$ gives a closed convex realization of $K_{n_1,n_2,\ldots,n_k}$.  See Figure \ref{fig:completegraphs} for examples of the constructions in the proof of (2).  
\begin{figure}[h]
\includegraphics[width=1.8in]{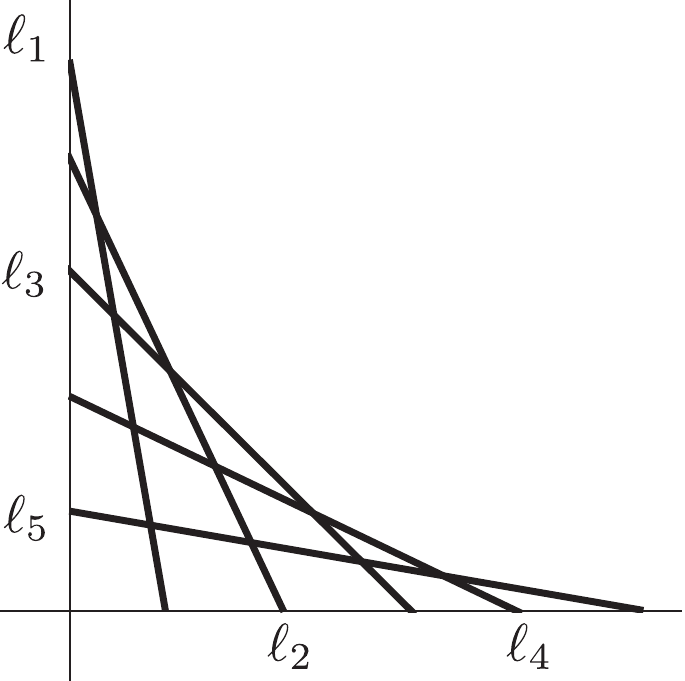}\hspace{1in} \includegraphics[width=2in]{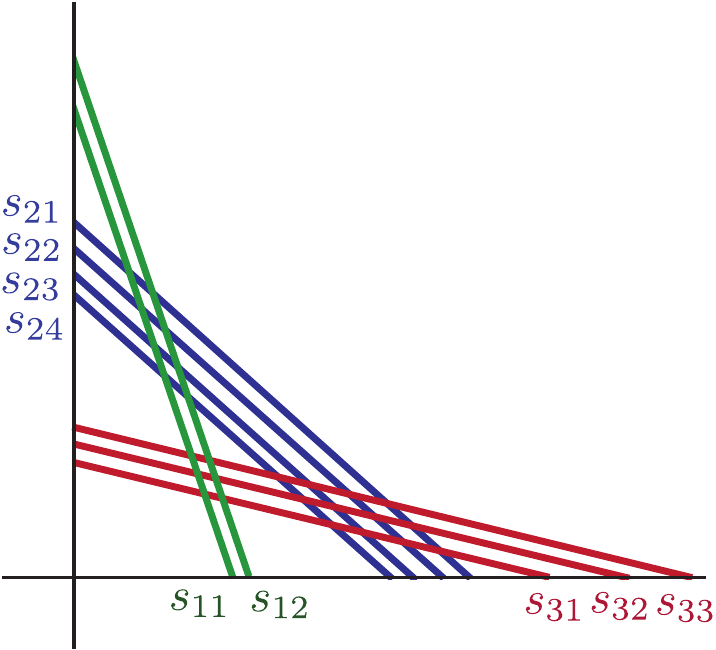} 
\caption{A closed convex realization of $K_5$ (left) and a closed convex realization of $K_{2,4,3}$ (right), as constructed in the proof of Proposition \ref{prop:R2} }\label{fig:completegraphs}
\end{figure}
It remains to prove (3).  Without loss of generality, we assume $N_G(u_k)=\{v_1,v_2,\ldots,v_r\}$, indexed in such a way that each set $N_G(u_j)$ is $\{v_1,v_2,\ldots,v_s\}$ for some $s$. To realize $G$, first start with a realization of $K_n$ as in the proof of (2), where $v_j$ is represented by $\ell_j$ for each $j$.  Now, extend each line segment $\ell_j$ for $1 \leq j \leq r$ so that $(0,j)$ remains as an endpoint, the slope remains the same, but the lower endpoint has $y$-coordinate $-k$.  Then, for each $s$ with $1 \leq s \leq k$, introduce a line segment $\ell'_s$ that lies on the line in the $xy$-plane given by $y=s$, and only intersects the line segments in the set $\{\ell'_j \ | \ j \in N_G(u_s)\}$.  The line segments $\ell_1,\ldots,\ell_n,\ell'_1,\ldots,\ell'_k$ give a closed realization of $G$.  See Figure \ref{fig:extendedcomplete} for an example of this construction.
\end{proof}

\begin{figure}[h]
\centering
\includegraphics[width=1.8in]{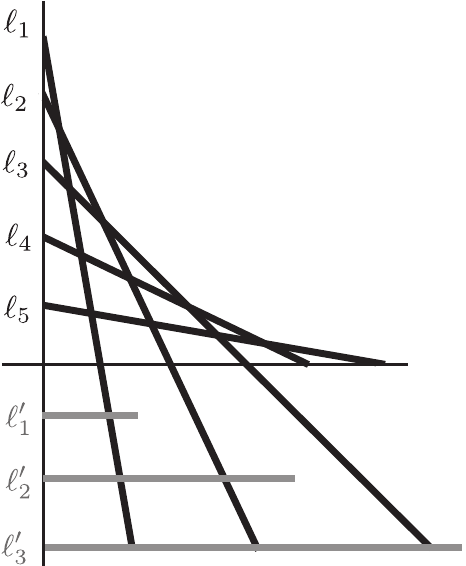}
\caption{A closed convex realization of the graph $G$  with vertices $v_1,v_2, v_3, v_4,v_5, u_1,u_2,u_3$ where the induced graph on $v_1,...,v_5$ is complete, and $N(u_3) = \{v_1,v_2,v_3\}$, $N(u_2) = \{v_1,v_2\}$ and $N(u_1) = \{v_1\}$.  }\label{fig:extendedcomplete}
\end{figure}

Thus far, we have exhibited classes of graphs that can be realized in $\R^2$, which included graphs $G_\C$ that are planar.  We know show how to augment any nonplanar graph to create a graph that can not be realized in $\R^2$.

\begin{proposition}\label{prop:nonplanar} Let $G$ be a nonplanar graph.  Let $G'$ be the graph obtained from $G$ by replacing each edge $v_i v_j$ by a length 2 path $v_i, v_{ij}, v_j$ (we refer to this as the \emph{edge subdivision} of $G$ throughout).  Then $G'$ does not have an open convex realization in $\R^2$, and hence its minimal embedding dimension is $3$.
\end{proposition}

\begin{proof}
Suppose by contradiction that $G'$ has an open convex realization in $\R^2$. Let the graph $G$ have vertex set $\{v_1,v_2,..,v_n\}$, so $G'$ has as its vertices $\{v_i \ | \ i=1,...,n\}$ together with vertices $\{v_{ij} \ | \ v_iv_j\in E(G)\}$ where for any $i,j$, $v_{ij}$ is adjacent only to $v_i$ and $v_j$.  Suppose the open convex realization $\mathcal{U}$ of $G'$ consists of the sets $\{U_i\}$ and $\{U_{ij}\}$ where for any $i$, $U_i$ is the open convex set corresponding to $v_i$, and for any $i \neq j$ with $v_iv_j\in E(G)$, $U_{ij}$ is the open convex set corresponding to $v_{ij}$.

Fix $i,j$ such that $v_i$ and $v_j$ are adjacent in $G$. Let $p_i$ and $p_j$ be points in $U_i$ and $U_j$ respectively, that do not lie in any other sets in $\U$. Since $U_i \cap U_{ij}$ and $U_j \cap U_{ji}$ are nonempty, we can find points $x_{ij}$ and $x_{ji}$ in $U_i \cap U_{ij}$ and $U_j \cap U_{ij}$, respectively. Let $x_{ij}x_{ji}$ intersect $\partial U_i$ and $\partial U_j$ at point $p_{ij}$ and $p_{ji}$, respectively.  Define the path $P_{ij}$ from $p_i$ to $p_j$ by concatenating the line segments $p_ip_{ij}$, $p_{ij}p_{ji}$, and $p_{ji}p_j$ in that order.

Now additionally fix indices $k,l$.  We claim that the two different paths $P_{ij}$, $P_{kl}$ can only intersect at the points $p_i$,$p_j$,$p_k$ or $p_l$, if anywhere.  To do so, it is enough to show that among any pair of line segments, one chosen from $\{p_ip_{ij},p_{ij}p_{ji},p_{ji}p_j\}$ and one from $\{p_kp_{kl},p_{kl}p_{lk},p_{lk}p_l\}$,  their intersection (if it exists), must be one of the points $p_i$,$p_j$,$p_k$ or $p_l$.  We split this into three cases:

First, consider the intersection of $p_ip_{ij}$ and $p_kp_{kl}$.  If $i=k$ then the two segments can only intersect at $p_i$, unless $j=l$, in which case the segments were the same segments to begin with. If $i \neq k$, then observe that $p_ip_{ij} \in U_i$, $p_kp_{kl} \in U_k$ and $U_i \cap U_k$ is empty because $v_i$ and $v_k$ are not adjacent in $G'$. A similar argument establishes our desired result when the pair of segments in question are $\{p_ip_{ij} , p_{kl}p_k\}$, $\{p_{ij}p_i, p_{kl}p_k\}$ and $\{p_{ij}p_i, p_{k}p_{kl}\}$. 

Second, consider the intersection of $p_{ij}p_{ji}$ and $p_{kl}p_{lk}$. Notice that $p_{ij}p_{ji} \subseteq U_{ij}$ and $p_{kl}p_{lk} \subseteq U_{kl}$. Since $v_{ij}$ and $v_{kl}$ are not adjacent in $G'$, $U_{ij} \cap U_{kl}$ is empty, so the two paths in question can not intersect.

Finally, consider the intersection of $p_{i}p_{ij}$ and $p_{kl}p_{lk}$. Suppose that $i=k$.  When $j=l$, the segments in question are $p_ip_{ij}$, $p_{ij}p_{ji}$ but these are from the same path $P_{ij}$ so we need not consider this situation. When $j \neq l$, $p_{i}p_{ij} \subseteq U_i \cup U_{ij}$, and $p_{il}p_{li} \subseteq U_{il}\backslash U_i$. Since $j \neq l$, $U_{ij} \cap U_{il} = \emptyset$, and hence $(U_i \cup U_{ij}) \cap (U_{il}\backslash U_i) = \emptyset$, so the two segments in question do not intersect.  A similar argument establishes the result when $j = l$.  It remains to establish the desired result when $i \neq l,k$. Suppose for a contradiction that $p_ip_{ij}$ intersects $p_{kl}p_{kl}$. Since $p_ip_{ij} \subseteq U_i \cup \partial U_i$, and $p_{kl}p_{lk} \subseteq U_{lk}$, this implies $(U_i \cup \partial U_i) \cap U_{lk}$ is nonempty. However, this is impossible because $U_i \cap U_{lk} = \emptyset$ (because $v_i$ and $v_{lk}$ are not adjacent in $G'$) and $\partial U_i \cap U_{lk} = \emptyset$.

The above argument establishes that two distinct paths $P_{ij},P_{kl}$ can only intersect at their endpoints.  Construct a graph $G''$ on the same vertex set as $G$ with two vertices adjacent precisely when they are adjacent in $G$, but with each edge $v_iv_j$ drawn precisely along the path $P_{ij}$.  The graph $G''$ is a planar embedding of $G$, contradicting $G$ is planar.
\end{proof}

\section{Future directions}
This paper initiated the program of studying $k$-sparse codes with a full characterization of the structure of $2$-sparse codes.  Section 2 was dedicated to a topological and analytic investigation of such codes in order to achieve a full characterization through Theorem~\ref{thm:R3}, which additionally told us that any realizable $2$-sparse code has minimal embedding dimension at most $3$.  Section 3 then began the study of differentiating $2$-sparse codes by embedding dimension through Proposition~\ref{prop:R2} and Proposition~\ref{prop:nonplanar}.  The most pressing questions are how these investigations generalize when $k>2$.

\begin{question} For a particular $k$, how can we characterize which k-sparse codes are realizable? More specifically, given a positive integer $\ell$, for which $k$-sparse codes is $d(\C)=\ell$?  
\end{question} 

In investigating the minimum embedding dimension of a $k$-sparse code, certain dimension bounds can be used.  For example, suppose $\C$ is a $k$-sparse code with $\Delta = \Delta(\C)$, and let $f_d(\Delta)$ be the number of codewords in $\Delta$ with support size $d+1$.  Then, by applying the Fractional Helly theorem, we find $k>f_d(\Delta)/\binom{n-1}{d}$; this was noted in \cite{Youngs2015}.  Similar to this, many known bounds rely solely on the combinatorial information in the code and in particular the simplicial complex $\Delta(\C)$.  However, in addressing the question of whether a $k$-sparse code is realizable at all, an investigation into the topology of the matter can provide insight beyond what is apparent from the combinatorics.  This is especially evident from the developments in Section~\ref{sec:topological}.  The key idea there was shifting from one realization of a code to another by shrinking or expanding sets.  The question then for $k$-sparse codes for $k>2$ is what analogous topological operations to realizations preserve the underlying code.

\begin{question}
Given a convex realization $\U=\{U_1,...,U_n\}$ of a code $\C$ in $\R^d$, what topological maps can be applied to the sets $U_i$ so that the resulting sets still form a convex realization of $\C$?
\end{question}

\section*{Acknowledgments}
\noindent The authors thank Dean's Office and the Department of Mathematics at Harvey Mudd College for their summer research support, and thank Carina Curto, Chad Giusti, Elizabeth Gross, Vladimir Itskov, Bill Kronholm and Anne Shiu for fruitful conversations.

\end{document}